\documentclass[12pt]{amsart}
\usepackage{amsfonts}
\usepackage{amsfonts,latexsym,rawfonts,amsmath,amssymb,amsthm}
\usepackage[plainpages=false]{hyperref}
\usepackage{cite}

\usepackage{graphicx}

\RequirePackage{color}

\numberwithin{equation}{section}

\newcommand{\beq}{\begin{equation}}
\newcommand{\eeq}{\end{equation}}
\newcommand{\beqs}{\begin{eqnarray*}}
\newcommand{\eeqs}{\end{eqnarray*}}
\newcommand{\beqn}{\begin{eqnarray}}
\newcommand{\eeqn}{\end{eqnarray}}
\newcommand{\beqa}{\begin{array}}
\newcommand{\eeqa}{\end{array}}

\newtheorem{definition}{Definition}
\newtheorem{lemma}{Lemma}
\newtheorem{remark}{Remark}

\newtheorem{theorem}{Theorem}
\newtheorem{example}{Example}
\newtheorem{corollary}{Corollary}

\title  {a note on multi-transitivity in non-autonomous discrete systems}

\begin{document}




\bibliographystyle{plain}

\maketitle

\baselineskip=15.8pt
\parskip=3pt

\centerline {\bf   \small Hongbo Zeng}
\centerline {School of Mathematics and Statistics, Changsha University of Science and Technology;}
\centerline {Hunan Provincial Key Laboratory of Mathematical Modeling and Analysis in Engineering,}
\centerline {Changsha 410114, Hunan, China}

\vskip20pt

\noindent {\bf Abstract}:
This paper is concerned with some stronger forms of transitivity in non-autonomous discrete systems$(f_{ 1,\infty})$ generated by a uniformly convergent sequence of continuous self maps. Firstly, we present two counterexamples to show that Theorem 3.1 obtained by Salman and Das in [Multi-transitivity in nonautonomous discrete systems Topol. Appl. 278(2020)107237] is not true.  Then, we introduce and study mildly mixing in non-autonomous discrete systems, which is stronger than mixing. We obtain that multi-transitivity implies Li-Yorke chaos and that mildly mixing implies multi-transitivity, which answer the open problems 1 and 2 in the paper above. Additionally, we give a counterexample which shows that Theorem 2.3 and Theorem 2.4 given by Sharma and Raghav in [On dynamics generated by a uniformly convergent sequence of maps Topol. Appl. 247 (2018)81-90] are both incorrect and give the correct proofs of them. Finally, some counterexamples are constructed justifying that some results related to stronger forms of transitivity which are true for autonomous systems but fail in non-autonomous systems, and establish a sufficient condition under which the results still hold in non-autonomous systems.

 \vskip20pt

\vskip20pt

\baselineskip=15.8pt
\parskip=3pt


\maketitle

\baselineskip=15.8pt
\parskip=3.0pt




\section{Introduction}

Dynamical systems theory is an effective mathematical mechanism which describes the time dependence of a point in a geometric space and has remarkable connections with different areas of mathematics such as topology and number theory. It is used to deal with the complexity, instability, or chaoticity in the real world, such as in meteorology, ecology, celestial mechanics, and other natural sciences. In recent years, more and more scholars have begun to devote themselves to the research in topological dynamical systems, and have achieved many significant results (see \cite{q12}).  topological transitivity (shortly, transitivity) has been an eternal topic in the study of topological dynamical systems, which is a crucial measure of system complexity. The concept of transitivity can be traced back to Birkhoff \cite{q12}. Since then, many studies have been devoted to this topic. Besides, as we all know, the chaos theory was first strictly defined by Li and Yorke in 1975 \cite{a1}. And sensitive dependence on initial conditions(briefly, sensitivity), first defined in \cite{q9}, is one of the most remarkable components of dynamical systems theory, which is closely linked to different variants of mixing. It characterizes the unpredictability in chaotic phenomena and is an integral part of different types of chaos. Afterwards, some stronger forms of sensitivity and  transitivity \cite{q11} were successively proposed by Moothathu.  The relationship among variations on the concept of transitivity can be found in \cite{q27}. The notion of multi-transitivity for autonomous discrete systems was introduced by Moothathu in 2010 \cite{q33}. He has proved that for minimal autonomous systems weakly mixing and multi-transitivity are equivalent. In 2012, Kwietniak and Oprocha have constructed examples showing that in general there is no relation between multi-transitivity and weakly mixing in autonomous discrete systems \cite{q37}.  Chen \cite{q38} discussed multi-transitivity with
respect to a vector and proved that multi-transitive system is Li-Yorke chaotic.  In 2019, authors have characterized multi-transitivity for autonomous discrete dynamical systems using different forms of shadowing and uniform positive entropy \cite{q39}.

\par
As a natural extension of autonomous discrete dynamical systems (Abbrev. ADS), non-autonomous discrete dynamical systems (NDS) are an important part of topological dynamical systems. Compared with classical dynamical systems (ADS), NDS can describe various dynamical behaviors more flexibly and conveniently. Indeed, most of the natural phenomena, whose behavior is influenced by external forces, are time dependent external forces. As a result, many of the methods, concepts and results of autonomous dynamical systems are not applicable. Therefore, there is a strong need to study and develop the theory of time variant dynamical systems, that is, non-autonomous dynamical systems. As a consequence, the techniques used in this context are, in general, different from those used for autonomous systems and make this discipline of great interest. In such systems the trajectory of a point is given by successive application of different maps. These systems are related to the theory of difference equations, and in general, they provide a more adequate framework for the study of natural phenomena that appear in biology, physics, engineering, etc. Meanwhile, the dynamics of non-autonomous discrete systems has became
an active research area, obtaining results on topological entropy, sensitivity, mixing properties, chaos, and other properties.

The notion of non-autonomous dynamical system was introduced by Kolyada and Snoha \cite{a14} in 1996. Since then, the study of complexity of non-autonomous dynamical systems has seen remarkable increasing interest of many researchers, see \cite{q16,q18,q14,q23,wx1,q25} and the references therein. Chaos and sensitivity for NDS were introduced by Tian and Chen  in 2006 \cite{q16}. Moreover, Cnovas\cite{a19} studied the limit behaviour of sequence. Recently, Salman and Das introduced and studied the notions of multi-transitivity and thick transitivity for
non-autonomous discrete dynamical systems\cite{wx1}. The authors obtained a sufficient condition under
which in minimal non-autonomous discrete dynamical systems, multi-transitivity,
thick transitivity and weakly mixing of all orders are equivalent. Sharma and Raghav\cite{q36} related the dynamical behavior of $(X, f_{ 1,\infty} )$ with the dynamical behavior
of the limiting system (and vice versa).  The interested reader in transitivity, chaos and sensitivity for NDS might consult \cite{q35,q23,q16,q25}.

\par
   This paper is organized as follows. In Section 2, we will state some preliminaries, definitions and some lemmas. The main conclusions will be given in Section 3. At first, we give two counterexamples which shows that Theorem 3.1 obtained by Salman and Das in \cite{wx1} is not true.  Next, for NDS, we prove that multi-transitivity implies Li-Yorke chaos, and that mildly mixing implies multi-transitivity, which answer the open problems 1 and 2 in the paper\cite{wx1}. Then, we provide a counterexample to show that Theorem 2.3 and Theorem 2.4 given by Sharma and Raghav in \cite{q36} do not hold and reprove them. Besides, we give a sufficient condition under which weakly mixing(total transitivity, respectively) is equivalent for NDS and the limiting system. Moreover, some counterexamples are presented justifying that some results related to stronger forms of transitivity which are true for autonomous systems but fail in non-autonomous systems, and establish a sufficient condition under which the results still hold in non-autonomous systems.  These counterexamples show that there is a significant
difference between the theory of ADS and the theory of NDS. At last, we remove the condition 'the space has no isolated point' in the Theorem 3.5(5)\cite{q29} and reprove it with a simpler method. Section 4 posts several open problems for future research.

\section{Preparations and lemmas}
In this section, we mainly give some different concepts of transitivity, chaos and sensitivity for NDS (see,for example,\cite{q14,q35,q18,q36}) and some lemmas required for remaining sections of the paper.

Assume that $\mathbb{N}=\{1,2,3,...\}$. Let $(X,d)$ be a compact metric
space and $f_n : X \rightarrow X$ be a sequence of continuous functions, where $n\in \mathbb{N}$. An non-autonomous discrete
dynamical systems (NDS) is a pair $(X, f_{ 1,\infty} )$ where $f_{ 1,\infty}= \{f_n\}_{n=1}^\infty$. For any $i,n\in \mathbb{N}$,
define the composition
$f_i^n:= f_{i+(n-1)} \circ \cdot\cdot\cdot \circ f_i$ ,
and usual $f_i^0=id_X$, where $id_X$ is the identity map on $X$. In particular, if $f_n = f$ for all $n \in \mathbb{N}$, the pair $(X, f_{ 1,\infty} )$ is
just the classical discrete system $(X, f)$. The orbit of
a point $x$ in $X$ is the set
$orb(x, f_{ 1,\infty}):=\{x,f_1^1(x),f_1^2(x),...,f_1^n(x),...\}$,
which can also be described by the difference equation $x_0 = x$ and $x_{n+1} = f_n (x_n )$. Given $k \in \mathbb{N}$, $f_{ 1,\infty}^{[k]}:=\{f_{k(n-1)+1}^k\}_{n=1}^\infty$ is called the $k^{th}$-iterate of NDS $(X, f_{ 1,\infty} )$. For non-autonomous discrete dynamical systems $(X, f_{n,\infty} )$ with $n\in\mathbb{N}$, the mapping sequence $f_{n,\infty}$ is defined by $f_{n,\infty}=\{f_n,f_{n+1},...\}$. Denote $f_1^{-n}=(f_1^{n})^{-1}=f_1^{-1}\circ f_2^{-1}\circ\cdots\circ f_n^{-1}$, and $B(x,\varepsilon)$ the open ball of radius $\varepsilon>0$ with center $x$. We denote by $int(A)$ the interior of $A$ and by $\overline{A}$ its closure.
If $\mathcal{C}(X)$ is the collection of continuous self maps on $X$, then for any $f,g\in\mathcal{C}(X)$, $D(f,g)=sup_{x\in X}d(f(x),g(x))$ is a metric on $\mathcal{C}(X)$ known as the $Supremum$ $metric$. A  collection of sequences $\{f_n^{n+k}\}_{k\in\mathbb{N}}$ converges collectively to $\{f^k\}_{k\in\mathbb{N}}$ with respect to the metric $D$ if for any $\varepsilon>0$, there exists $r_0\in\mathbb{N}$ such that $D(f_r^k,f^k) < \varepsilon$ for all $r\geq r_0$ and for every $k\in\mathbb{N}$\cite{q36}.
We always suppose that $(X, f_{ 1,\infty} )$  is a non-autonomous discrete system and that all the maps are continuous from $X$ to $X$ in the following context.

\begin{definition}
The non-autonomous system $(X, f_{ 1,\infty} )$  is said to be (topologically) transitive, if for any two non-empty subsets $U,V\subseteq X$, there exists $n\in \mathbb{N}$ such that $f_1^n(U)\cap V\neq\emptyset$.

The non-autonomous system $(X, f_{ 1,\infty} )$  is said to be weakly mixing, if for any four non-empty subsets $U_1,U_2,V_1,V_2\subseteq X$, there exists $n\in \mathbb{N}$ such that $f_1^n(U_1)\cap V_1\neq\emptyset$ and $f_1^n(U_2)\cap V_2\neq\emptyset$.

The non-autonomous system $(X, f_{ 1,\infty} )$  is said to be weakly mixing of order $n(n\geq2)$, if for any collection of non-empty subsets $U_1,U_2,...,U_n,V_1,V_2,...,V_n\subseteq X$, there exists $m\in \mathbb{N}$ such that $f_1^m(U_i)\cap V_i\neq\emptyset$ for each $i=\{1,2,...,n\}$.

The non-autonomous system $(X, f_{ 1,\infty} )$  is said to be mixing, if for any two non-empty subsets $U,V\subseteq X$, there exists $N\in \mathbb{N}$ such that $f_1^n(U)\cap V\neq\emptyset$ for any $n\geq N$.

The non-autonomous system $(X, f_{ 1,\infty} )$  is said to be mildly mixing, if for any transitive non-autonomous system $(Y, g_{ 1,\infty} )$, the product system $(X\times Y, f_{ 1,\infty}\times g_{ 1,\infty} )$ is transitive.

The non-autonomous system $(X, f_{ 1,\infty} )$  is said to be totally transitive, if for any $n\in \mathbb{N}$, $f_{ 1,\infty}^{[n]}$ is transitive.

The non-autonomous system $(X, f_{ 1,\infty} )$  is said to be strongly transitive, if for any non-empty subset $U\subseteq X$, there exists $M\in \mathbb{N}$ such that $\cup_{i=1}^Mf_1^i(U)=X$.

\end{definition}

\begin{definition}
 A point $x\in X$ is said to be transitive if the orbit of $x$ is dense in $X$. The non-autonomous system $(X, f_{ 1,\infty} )$  is said to be minimal if all the points of $X$ are transitive. The non-autonomous system $(X, f_{ 1,\infty} )$  is said to be feeble open if $int(f_i(U))\neq\emptyset$ for any nonempty open set $U\subseteq X$ and any $i\in \mathbb{N}$.
\end{definition}

\begin{definition}
 A non-autonomous discrete system $f_{ 1,\infty}$ is called multi-transitive if $f_{ 1,\infty}\times f_{ 1,\infty}^{[2]}\times\cdots\times f_{ 1,\infty}^{[m]}:X^m\rightarrow X^m$ is transitive for any $m\in \mathbb{N}$. Equivalently, if for any collection of nonempty open subsets $U_1,U_2,...,U_m;V_1,V_2,...,V_m$ of $X$, there exists $l\in \mathbb{N}$ such that  $f_1^{jl}(U_j)\cap V_j\neq\emptyset$ for each $j\in\{1,2,...,m\}$  and for any $m\in \mathbb{N}$, then $f_{ 1,\infty}$ is multi-transitive.
\end{definition}

For the convenience of the following context, we denote $$N_{f_{ 1,\infty}}(U,V)=\{n\in \mathbb{N} \mid f_1^n(U)\cap V\neq\emptyset\},$$ and denote $$N_{f_{1,\infty}}(U,\delta)=\{n\in\mathbb{N}\mid \text{there exist} \ x,y\in U \ \text{such that } d(f_1^n(x),f_1^n(x))>\delta \}$$ for any nonempty open sets $U,V$ of $X$.

\begin{definition}
A set $A\subseteq \mathbb{N}$ is called syndetic if there exists a positive integer $M$ such that $\{i,i+1,...,i+M\}\cap A\neq\emptyset$ for every $i\in \mathbb{N}$, i.e. it has bounded gaps. A set $A\subseteq \mathbb{N}$ is called cofinite if there exists $N\in \mathbb{N}$ such that $A\supseteq[N,\infty)\cap \mathbb{N}$. A set $A\subseteq \mathbb{N}$ is called thick if it contains arbitrarily long runs of positive integers, that is, for any $p\in \mathbb{N}$, there exists some $n\in \mathbb{N}$ such that $\{n,n+1,...,n+p\}\subseteq A$.

The non-autonomous system $(X, f_{ 1,\infty} )$  is said to be syndetically transitive, if for any two non-empty subsets $U,V\subseteq X$, the set $N_{f_{ 1,\infty}}(U,V)$ is syndetical.

The non-autonomous system $(X, f_{ 1,\infty} )$  is said to be sensitive, if there is $\delta>0$ such that for any nonempty open set $U\subset X$, there exist $x,y\in U$ and $n\in \mathbb{N}$ such that $d(f_1^n(x),f_1^n(y))>\delta$.

A non-autonomous system $(X, f_{ 1,\infty} )$ is said to be syndetically sensitive, if there exists $\delta>0$ such that for any nonempty open set $U\subset X$, $N_{f_{1,\infty}}(U,\delta)$ is a syndetic set.

 A non-autonomous system $(X, f_{ 1,\infty} )$ is said to be thickly sensitive, if there exists $\delta>0$ such that for any nonempty open set $U\subset X$, $N_{f_{1,\infty}}(U,\delta)$ is thick.

A non-autonomous system $(X, f_{ 1,\infty} )$ is said to be multi-sensitive, if there exists $\delta>0$ such that for any $m\in \mathbb{N}$ and any nonempty open sets $U_1,U_2,...,U_m\subset X$, $\bigcap_{i=1}^mN_{f_{1,\infty}}(U_i,\delta)\neq\emptyset$, where $\delta$ is called constant of sensitivity.
\end{definition}

\begin{definition}
 A point $x\in X$ is called $k$-periodic if $f_1^{kn}(x)=x$ for any $n\in \mathbb{N}$. A point $x\in X$ is called almost periodic point if for any $\varepsilon>0$, the set $\{n\in \mathbb{N} \mid d(f_1^n(x),x)<\varepsilon\}$ is syndetical.
\end{definition}

\begin{definition}
 If there exists an uncountable subset $S\subseteq X$ such that for any different points $x,y\in S$  we have
$$\liminf_{n\rightarrow\infty} d(f_1^{n}(x), f_1^{n}(y))=0,\ \limsup_{n\rightarrow\infty} d(f_1^{n}(x), f_1^{n}(y))>0,$$
then $f_{ 1,\infty}$ is said to be Li-Yorke chaos.
\end{definition}

\begin{lemma}(\cite[Corollary1]{q299})\label{lychaos}
Let $\{p_k\}_{k=1}^{\infty}$ be a sequence of positive integers, $\{A_i\}_{i=0}^{\infty}$ and $\{B_i\}_{i=0}^{\infty}$ be decreasing sequences of compact sets satisfying $$\bigcap_{i=0}^{\infty}A_i={a}, \bigcap_{i=0}^{\infty}B_i={b},$$ where $a\neq b$. If $\forall c=C_1C_2...$, where $C_k=A_k$ or $B_k$ for $k=1,2...$, there exists $x_c\in X$, such that $\forall k\geq1$, we have $f_1^{p_k}(x_c)\in C_k$, then $f_{ 1,\infty}$ is Li-Yorke chaos.
\end{lemma}

\begin{lemma}(\cite[Lemma2.6]{q14})\label{yinlijmaa2}
Let $(X,d)$ be a metric space without isolated points. If $(X, f_{ 1,\infty} )$ is transitive, then for any pair of non-empty open subsets $U,V$ of $X$, the set $N_{f_{ 1,\infty}}(U,V)$ is infinite.
\end{lemma}

\begin{lemma}(\cite[Lemma3]{q299})\label{yizhi}
Assume that non-autonomous discrete system $(X,f_{1,\infty})$ converges uniformly to a map $f$. Then for any $\varepsilon >0$ and any $k\in \mathbb{N}$, there exist $\xi$ such that for any $n\in\mathbb{N}$ and any pair $x,y \in X$ with $d(x,y)<\xi$, $d(f_n^k(x),f_n^k(y))<\frac{\varepsilon}{2}$.
\end{lemma}

\section{Main results}
Salman and Das have proved the following results in Theorem 3.1\cite{wx1}.

\begin{theorem}\cite[Theorem 3.1]{wx1}\label{th31}
 Let $(X,f_{1,\infty})$ be a non-autonomous dynamical system such that each $f_i$ is surjective. If $(X,f_{1,\infty})$ is multi-transitive, then so is $(X,f_{k,\infty})$, for every $k\geq1$. Converse is true if the family $f_{1,\infty}$ is feebly open.
\end{theorem}
 However, we will disprove it. The following two examples justify the above-mentioned Theorem  \ref{th31}  need not be true for general non-autonomous dynamical system.

\begin{example}\label{li1}

Let $\Sigma_2=\{0,1\}^{\mathbb{Z}}=\{(...,x_{-2},x_{-1},\boxed{x_0},x_1,x_2,...)\mid x_i\in\{0,1\}, \text{for each} \\ i\in\mathbb{Z}\}$ with metric
$$d(x,y)=\sum\limits_{i=-\infty}\limits^{\infty}\frac{|x_i-y_i|}{2^{|i|}}$$
for any pair $x=(...,x_{-2},x_{-1},\boxed{x_0},x_1,x_2,...),y=(...,y_{-2},y_{-1},\boxed{y_0},y_1,y_2,...)\in\Sigma_2$.
Define the shift map $\sigma:\Sigma_2\rightarrow \Sigma_2$ by $\sigma(x)=(...,x_{-2},x_{-1},x_0,\boxed{x_1},x_2,...)$, where $x=(...,x_{-2},x_{-1},\boxed{x_0},x_1,x_2,...)\in\Sigma_2$. Consider the non-autonomous discrete system $(\Sigma_2,f_{1,\infty})$, where $f_{1,\infty}$ is defined by
$$f_{1,\infty}=\{id,id,\sigma,\sigma^{-1},\sigma^2,\sigma^{-2},...,\sigma^n,\sigma^{-n},...\}.$$
Then $$f_{2,\infty}=\{id,\sigma,\sigma^{-1},\sigma^2,\sigma^{-2},...,\sigma^n,\sigma^{-n},...\}.$$
It is easy to see that $f_{1,\infty}$ is feebly open and each $f_i$ is surjective.

Now we show that $f_{2,\infty}$ is multi-transitive. For any $m\in\mathbb{N}$, let $U_1,U_2,...,U_m;V_1,V_2,...,V_m$ be any collection of nonempty open subsets of $X$. Since $\sigma$ is mixing, there exists $M\in\mathbb{N}$ such that $\sigma^p(U_j)\cap V_j\neq\emptyset$ for any $p>M$ and any $j\in\{1,2,...,m\}$. Take $l=2(M+1)$, then we can get that $f_2^{lj}=f_2^{2(M+1)j}=\sigma^{(M+1)j}$. Therefore,
$$f_2^{lj}(U_j)\cap V_j=\sigma^{(M+1)j}(U_j)\cap V_j\neq\emptyset$$ for any $j\in\{1,2,...,m\}$. Thus $f_{2,\infty}$ is multi-transitive.

Next, we prove that $f_{1,\infty}$ is not multi-transitive. Take $m=2$, and take nonempty open subsets $U_1,U_2,V_1,V_2$ of $X$ such that $U_1\cap V_1=\emptyset$ and $U_2\cap V_2=\emptyset$. Note that for any $t\in\mathbb{N}$, $f_1^{2t}\equiv id$. Then for any $t\in\mathbb{N}$, $f_1^{2t}(U_2)\cap V_2=U_2\cap V_2=\emptyset$. Therefore, $f_{1,\infty}$ is not multi-transitive.

\end{example}

\begin{example}

Let $\Sigma_2$ and $\sigma$ be defined in Example \ref{li1}. Consider the non-autonomous discrete system $(\Sigma_2,f_{1,\infty})$, where $f_{1,\infty}$ is defined by
$$f_{1,\infty}=\{id,\sigma,\sigma^{-1},\sigma^2,\sigma^{-2},...,\sigma^n,\sigma^{-n},...\}.$$
Then $$f_{2,\infty}=\{\sigma,\sigma^{-1},\sigma^2,\sigma^{-2},...,\sigma^n,\sigma^{-n},...\}.$$
It is not difficult to see that $f_{1,\infty}$ is feebly open and each $f_i$ is surjective.

Now we verify that $f_{1,\infty}$ is multi-transitive. For any $m\in\mathbb{N}$, let $U_1,U_2,...,U_m;V_1,V_2,...,V_m$ be any collection of nonempty open subsets of $X$. Since $\sigma$ is mixing, there exists $M\in\mathbb{N}$ such that $\sigma^p(U_j)\cap V_j\neq\emptyset$ for any $p>M$ and any $j\in\{1,2,...,m\}$. Take $l=2(M+1)$, we can get that $f_1^{lj}=f_1^{2(M+1)j}=\sigma^{(M+1)j}$. Therefore,
$$f_1^{lj}(U_j)\cap V_j=\sigma^{(M+1)j}(U_j)\cap V_j\neq\emptyset$$ for any $j\in\{1,2,...,m\}$. Thus $f_{1,\infty}$ is multi-transitive.

Next, we claim that $f_{2,\infty}$ fails to be multi-transitive. Take $m=2$, and take nonempty open subsets $U_1,U_2,V_1,V_2$ of $X$ such that $U_1\cap V_1=\emptyset$ and $U_2\cap V_2=\emptyset$. Observe that for any $t\in\mathbb{N}$, $f_1^{2t}\equiv id$. Then for any $t\in\mathbb{N}$, $f_2^{2t}(U_2)\cap V_2=U_2\cap V_2=\emptyset$. Therefore, $f_{2,\infty}$ fails to be multi-transitive.

\end{example}

\begin{remark}
We can use a similar argument to show that for every $k\geq2$ there exist $(X,f_{1,\infty})$ and $(X,f_{k,\infty})$ such that the former is multi-transitive(not multi-transitive, respectively) but the later is not multi-transitive(multi-transitive, respectively).
\end{remark}

The next example to show that for general system there exist autonomous dynamical systems $(X,f)$ and non-autonomous discrete systems $(X,f_{1,\infty})$ such that $f$ is nonminimal even though $(X,f_{1,\infty})$ is minimal, feeble open, and converges uniformly to a map $f$ with $\{f_n^{n+k}\}_{k\in\mathbb{N}}$ converging collectively to $\{f^k\}_{k\in\mathbb{N}}$.

\begin{example}\label{li4}

Let $(X=\{1,2\},d)$ be a discrete metric space with the discrete metric $d$. And let $f(x)\equiv2$, $f_1(x)=1$ and $f_i(x)\equiv2$ for all $i\geq2$ for any $x\in X$. Then it is easy to see that $(X,f_{1,\infty})$ is minimal, feeble open, and converges uniformly to a map $f$ with $\{f_n^{n+k}\}_{k\in\mathbb{N}}$ converging collectively to $\{f^k\}_{k\in\mathbb{N}}$, but $f$ is nonminimal.

\end{example}

For transitive non-autonomous discrete system, if $f_{1,\infty}$ has no isolate point, then the set $N_{f_{1,\infty}}(U,V)$ is infinite\cite[Lemma2.6]{q14}, but the result is not true in general if $f_{1,\infty}$ has isolate point(see, Example \ref{li4}). However, the following two theorems show that for any multi-transitive or weakly mixing non-autonomous discrete system, the result is true.
\begin{theorem}\label{mildlyinf}
Let $(X,f_{1,\infty})$ be a non-autonomous discrete system. If
$f_{1,\infty}$ is multi-transitive, then the set $\{l\in\mathbb{N}:f_1^{lj}(U_j)\cap V_j\neq\emptyset,\forall j=1,2,...,m\}$ for any collection of nonempty open subsets $U_1,U_2,...,U_m;V_1,V_2,...,V_m$ of $X$, is infinite.
\end{theorem}

\begin{proof}
Let us suppose that it is finite and let $p=max\{l\in\mathbb{N}:f_1^{lj}(U_j)\cap V_j\neq\emptyset,\forall j=1,2,...,m\}$. Put $$U_{2(i-1)p+j}'=U_j,V_{2(i-1)p+j}'=V_j$$ where $i=1,2,...,m,j=1,2,...,2p$.
Since $f_{1,\infty}$ is multi-transitive, for the above nonempty open subsets $U_1',U_2',...,U_{2mp}';V_1',V_2',...,V_{2mp}'$, there exist $k\in\mathbb{N}$ such that $f_1^{kh}(U_h')\cap V_h'\neq\emptyset$ for each $h=1,2,...,2mp$. It follows that $f_1^{2kpi}(U_j)\cap V_j\neq\emptyset$ for each $j=1,2,...,m$. Therefore $2kp\in \{l\in\mathbb{N}:f_1^{lj}(U_j)\cap V_j\neq\emptyset,\forall j=1,2,...,m\}$.
However, $2kp>p$, which is a contradiction to the definition of $p$.

\end{proof}

\begin{theorem}\label{weaklyinf}
Let $(X,f_{1,\infty})$ be a non-autonomous discrete system. If
$f_{1,\infty}$ is weakly mixing, then the set $\{n\in\mathbb{N}:f_1^{n}(U_j)\cap V_j\neq\emptyset,\forall j=1,2\}$ for any collection of nonempty open subsets $U_1,U_2;V_1,V_2$ of $X$, is infinite.
\end{theorem}

\begin{proof}
Firstly, we show that $X$ has no isolated point. If $X$ contains two isolated points $a,b$ at least. Then take $U_1=U_2=V_1=\{a\},V_2=\{b\}$. Since $f_{1,\infty}$ is weakly mixing, there exist $n\in\mathbb{N}$ such that $f_1^{n}(U_j)\cap V_j\neq\emptyset,\forall j=1,2$. That is, $f_1^{n}(a)=a,f_1^{n}(a)=b$, which is a contradiction.
If $X$ contains only one isolated point $a$, then take $U_1=U_2=V_1=\{a\}$, and let $V_2$ be a nonempty open subset of $X$ such that $a\notin V_2$. Since $f_{1,\infty}$ is weakly mixing, there exist $n\in\mathbb{N}$ such that $f_1^{n}(U_j)\cap V_j\neq\emptyset,\forall j=1,2$. That is, $f_1^{n}(a)=a,f_1^{n}(a)=b$, where $b\in V_2$, which is a contradiction. Therefore, $X$ has no isolated point.

Next, suppose that the set is finite and let $p=max\{n\in\mathbb{N}:f_1^{n}(U_j)\cap V_j\neq\emptyset,\forall j=1,2\}$.
Since $X$ has no isolated point, $V_1$ and $V_2$ contain infinitely many points, thus we can take a collection $V_1^1,V_1^2,...,V_1^{p+1}\subseteq V_1$ and $V_2^1,V_2^2,...,V_2^{p+1}\subseteq V_2$ such that them are mutually disjoint open sets. Since $p\in\{n\in\mathbb{N}:f_1^{n}(U_j)\cap V_j\neq\emptyset,\forall j=1,2\}$, there exist $x_1\in U_1$ and $x_2\in U_2$ such that $f_1^{p}(x_1)\in V_1$ and $f_1^{p}(x_2)\in V_2$. Now, there exist $j_1,j_2\in\{1,2,...,p+1\}$ such that $f_1^{i}(x_1)\notin V_{j_1}$ and $f_1^{i}(x_2)\notin V_{j_2}$ for all $i=1,2,...,p$ due to the choice of $V_1^1,V_1^2,...,V_1^{p+1},V_2^1,V_2^2,...,V_2^{p+1}$. By continuity, there are open neighborhoods $U_1'$ of $x_1$ contained in $U_1$ and $U_2'$ of $x_2$ contained in $U_2$ such that
\begin{equation}\label{weakly1}
f_1^{i}(U_1')\cap V_{j_1}=\emptyset \quad\text{and} \quad f_1^{i}(U_2')\cap V_{j_2}=\emptyset, \quad\forall i=1,2,...,p.
\end{equation}
Since $f_{1,\infty}$ is weakly mixing, there exist $m\in\mathbb{N}$ such that
$$f_1^{m}(U_1')\cap V_{j_1}\neq\emptyset\quad\text{and} \quad f_1^{m}(U_2')\cap V_{j_2}\neq\emptyset.$$
Further,
$$f_1^{m}(U_1)\cap V_{1}\neq\emptyset\quad\text{and} \quad f_1^{m}(U_2)\cap V_{2}\neq\emptyset.$$
By (\ref{weakly1}), we have $m>p$, which is a contradiction to the definition of $p$.

\end{proof}

The next result answers open problem 1 of \cite{wx1}.
\begin{theorem}
Let $(X,f_{1,\infty})$ be a minimal non-autonomous system with $f_{1,\infty}$ being feeble open and surjective and converging uniformly to a map $f$. If $\{f_n^{n+k}\}_{k\in\mathbb{N}}$ converges collectively to $\{f^k\}_{k\in\mathbb{N}}$, then $(X,f_{1,\infty})$ is multi-transitive $\Rightarrow$ $(X,f_{1,\infty})$ is Li-Yorke chaos.
\end{theorem}
\begin{proof}
Let $a,b\in X$ with $a\neq b$. Take any a nonempty open set $U_0\subseteq X$ such that $\overline{U_0}$ is compact. Since $(X,f_{1,\infty})$ is multi-transitive and minimal, so by \cite[Theorem 4.1]{wx1} $f_{1,\infty}$ is weakly mixing of all orders. Thus, there exists $p_1>0$ such that
$$f_1^{p_1}(U_0)\cap B(a,1)\neq\emptyset \quad\text{and}\quad f_1^{p_1}(U_0)\cap B(b,1)\neq\emptyset .$$
Thus, we can find points $x_1,x_2\in U_0$ such that $f_1^{p_1}(x_1)\in B(a,1),f_1^{p_1}(x_2)\in B(b,1)$. Now, assume that there exist positive integers $p_1<p_2<\cdots<p_k$ such that for each finite sequence $A_1A_2\cdots A_k$, where $A_i\in\{B(a,\frac{1}{i}),B(b,\frac{1}{i})\}(i=1,2,...,k)$, there is a point $x\in U_0$ satisfying $f_1^{p_i}(x)\in A_i$ for $i=1,2,...,k$. Denote by $S_k$ the set of all such points. By continuity of $f_{1,\infty}$, each $x\in S_k$ has an open nonempty neighborhood $W_x\subseteq U_0$ such that $f_1^{p_i}(W_x)\subseteq A_i$. Since $f_{1,\infty}$ is weakly mixing, so by Theorem \ref{weaklyinf}, there exists $p_{k+1}>p_k$ such that for each $x\in S_k$,
$$f_1^{p_{k+1}}(W_x)\cap B(a,\frac{1}{k+1})\neq\emptyset\quad \text{and}\quad f_1^{p_{k+1}}(W_x)\cap B(b,\frac{1}{k+1})\neq\emptyset .$$
Thus, by induction, we know that there exists a sequence of positive integers $\{p_k\}_{k=1}^{\infty}$ such that for any finite sequence $A_1A_2\cdots A_k$, there exists a point $x\in U_0$ satisfying $f_1^{p_i}(x)\in A_i$ for all $i=1,2,...,k$. Let $C=C_1C_2\cdots$ be an infinite sequence, where $C_k\in\{\overline{B(a,\frac{1}{k})},\overline{B(b,\frac{1}{k})}\}$. For each $k$, we can take a point $x_k\in U_0$ such that $f_1^{p_i}(x_k)\in C_i$ for all $i=1,2,...,k$. Since $\overline{U_0}$ is compact, the infinite sequence $\{x_k\}_{k=1}^{\infty}$ has a limit point in $\overline{U_0}$, say $x_c$. It is not difficult to show that for each $k=1,2,...$ $f_1^{p_k}(x_c)\in C_k$. Thus, by lemma \ref{lychaos}, $(X,f_{1,\infty})$ is Li-Yorke chaos.

\end{proof}

It is well known that mixing is strictly stronger than mildly mixing for autonomous discrete dynamical system.  However, the theorem below shows that the result is just the opposite for non-autonomous discrete dynamical system, which is interesting and magical.
\begin{theorem}\label{mild}
Let $(X,f_{1,\infty})$ be a non-autonomous system. If $f_{1,\infty}$ is mildly mixing, then $f_{1,\infty}$ is mixing. Converse is true if we consider the system without isolated points.
\end{theorem}

\begin{proof}
Suppose that $f_{1,\infty}$ is mildly mixing. If $f_{1,\infty}$ is not mixing, then there exist two nonempty open sets $U,V$ of $X$ such that for any $N\in\mathbb{N}$ there is some $n>N$ satisfying $f_1^n(U)\cap V=\emptyset$. Further, there exists a sequence of positive integers $\{n_k\}_{k=1}^{\infty}$ such that $f_1^{n_k}(U)\cap V=\emptyset$ and $n_{k+1}-n_k>2$ for any $k\in\mathbb{N}$. Besides, we can assume that the above $U,V$ satisfy $U\cap V=\emptyset$, if not, we can choose the nonempty open sets $U'\subseteq U,V'\subseteq V$ such that $U'\cap V'=\emptyset$, denote $U=U',V=V'$. Let $\Sigma_2$ and $\sigma$ be defined in Example \ref{li1}. Consider the non-autonomous discrete systems $(\Sigma_2,g_{1,\infty})$, where $g_{1,\infty}$ is defined by
\[g_i=\begin{cases}
\sigma^{n_k}, &  \text{if}\quad i=n_k, \\
\sigma^{-n_k}, &  \text{if}\quad i=n_k+1,  \\
id, &  else,
\end{cases}\]
where $ k\in \mathbb{N}$, that is,
$$g_{1,\infty}=\{id,id,...,id,\sigma^{n_1},\sigma^{-n_1},id,id,...,id,\sigma^{n_2},\sigma^{-n_2},id,id,...,id,\sigma^{n_k},\sigma^{-n_k},id,...\}.$$
With the similar argument in Example \ref{li1}, $g_{1,\infty}$ is transitive. However, we have that
\[\begin{cases}
f_1^{m}(U)\cap V=\emptyset, &  if\quad m\in\{n_k\}_{k=1}^{\infty}, \\
g_1^{m}(U)\cap V=\emptyset, &  if\quad m\notin\{n_k\}_{k=1}^{\infty}.
\end{cases}\]
Then one can get that $f_{1,\infty}^m\times g_{1,\infty}^m(U\times U,V\times V)=\emptyset$ for any $m\in\mathbb{N}$. Therefore, $f_{1,\infty}\times g_{1,\infty}$ is not transitive, that is, $f_{1,\infty}$ is not mildly mixing, which is a contradiction to the assumption.

Conversely, since the system has no isolated point, so by lemma \ref{yinlijmaa2}, the set $N_{g_{1,\infty}}(U,V)$ is infinite. So it is clear that mixing implies mildly mixing by definitions.

For general compact metric space, let $g_{1,\infty}$ be non-autonomous discrete system defined in Example \ref{li4}. Then it is easy to see that $g_{1,\infty}$ is transitive. Let $f_{1,\infty}=\{f,f,f,...\}$, where $f$ is mixing in autonomous discrete system. Then $f_{1,\infty}$ is also mixing. But it is obvious that $f_{1,\infty}\times g_{1,\infty}$ is not mildly mixing.
\end{proof}

A straightforward consequence of the previous result is the following corollary, which answers open problem 2 in \cite{wx1}.

\begin{corollary}
Let $(X,f_{1,\infty})$ be a non-autonomous discrete system. If
$f_{1,\infty}$ is mildly mixing, then $f_{1,\infty}$ is multi-transitive.
\end{corollary}

The example below disproves Theorem 2.3 and Theorem 2.4 given by Sharma and Raghav\cite{q36}. It is to be noted that the proof of \cite[Theorem 4.1]{wx1} cited the \cite[Theorem 2.3]{q36}. We add a condition to obtain the results and reprove them in the following two theorems.
\begin{theorem}\cite[Theorem 2.3]{q36}
Let $(X,f_{1,\infty})$ be a non-autonomous system with $f_{1,\infty}$ being feeble open and surjective and converging uniformly to a map $f$. If $\{f_n^{n+k}\}_{k\in\mathbb{N}}$ converges collectively to $\{f^k\}_{k\in\mathbb{N}}$, then $(X,f)$ is minimal $\Leftrightarrow$ $(X,f_{1,\infty})$ is minimal.
\end{theorem}
\begin{theorem}\cite[Theorem 2.4]{q36}
Let $(X,f_{1,\infty})$ be a non-autonomous system with $f_{1,\infty}$ being feeble open and surjective and converging uniformly to a map $f$. If $\{f_n^{n+k}\}_{k\in\mathbb{N}}$ converges collectively to $\{f^k\}_{k\in\mathbb{N}}$, then $(X,f)$ is transitive $\Leftrightarrow$ $(X,f_{1,\infty})$ is transitive.
\end{theorem}
\begin{example}\label{li5}

Let $(X=\{1,2,3\},d)$ be a discrete metric space with the discrete metric $d$. And let
\[f(x)=\begin{cases}
2, &   \text{if}\quad x=1, \\
3, &   \text{if}\quad x=2, \\
1, &   \text{if}\quad x=3,
\end{cases}\]
Consider the non-autonomous discrete system $(X,f_{1,\infty})$, where $f_{1,\infty}$ is defined by
$$f_{1,\infty}=\{f,f,f,id,id,id,...,id,...\}.$$
It is not difficult to see that $f_{1,\infty}$ is feebly open and surjective and that $f_{1,\infty}$ converges uniformly to the map $id$ and $\{f_n^{n+k}\}_{k\in\mathbb{N}}$ converges collectively to $\{id^k\}_{k\in\mathbb{N}}$.

It is easy to see that $f_{1,\infty}$ is minimal and transitive but $id$ is neither minimal nor transitive.  Moreover, for the non-empty open subsets $U=\{1\},V=\{2\}$ of $X$ and the point $1\in X$, the sets $\{n: f_1^n(1)\in V\}=\{1\}$ and $N_{f_{1,\infty}}(U,V)=\{1\}$ are both finite.

\end{example}

\begin{theorem}\label{min}
Suppose that $X$ has no isolate point. Let $(X,f_{1,\infty})$ be a non-autonomous system with $f_{1,\infty}$ being feeble open and surjective and converging uniformly to a map $f$. If $\{f_n^{n+k}\}_{k\in\mathbb{N}}$ converges collectively to $\{f^k\}_{k\in\mathbb{N}}$, then $(X,f)$ is minimal $\Leftrightarrow$ $(X,f_{1,\infty})$ is minimal.
\end{theorem}

\begin{proof}
Necessity is proved in Theorem 2.3 by Sharma and Raghav\cite{q36}, we will prove sufficiency.
Let $(X,f_{1,\infty})$ be minimal and let $x\in X$. Let $y\in X$ and let $\varepsilon>0$ be fixed. As $\{f_n^{n+k}\}_{k\in\mathbb{N}}$ converges collectively to $\{f^k\}_{k\in\mathbb{N}}$, there exists $r_0\in\mathbb{N}$ such that $D(f_r^k,f^k) < \frac{\varepsilon}{2}$ for all $r\geq r_0$ and for every $k\in\mathbb{N}$, which implies $d(f_r^k(z),f^k(z)) < \frac{\varepsilon}{2}$ for all $r\geq r_0$ and for every $k\in\mathbb{N}$ and any $z\in X$. Since $(X,f_{1,\infty})$ is minimal, orbit of any $z\in f_1^{-r_0}(x)$(under $f_{1,\infty}$) is dense in $X$ and hence intersects $B(y,\frac{\varepsilon}{2})$. Next, we claim that the set $\{n\in\mathbb{N}\mid f_1^n(z)\in B(y,\frac{\varepsilon}{2})\}$ is infinite. If not, let $m=max\{n\in\mathbb{N}\mid f_1^n(z)\in B(y,\frac{\varepsilon}{2})\}$. Since $X$ has no isolate point, there exist $m+1$ non-empty open sets $U_1,U_2,...,U_{m+1}$ in $B(y,\frac{\varepsilon}{2})$ such that $U_i\cap U_j=\emptyset$ for any $i,j\in\{1,2,...,m+1\}$ with $i\neq j$. Since $(X,f_{1,\infty})$ is minimal, there exists $n_i\in \mathbb{N}$ such $f_1^{n_i}(z)\in U_i\subseteq B(y,\frac{\varepsilon}{2})$ for any $i\in\{1,2,...,m+1\}$. Besides, we can get that $n_i\neq n_j $ for any $i,j\in\{1,2,...,m+1\}$ with $i\neq j$. Hence, there exists $n_i>m$ such that $f_1^{n_i}(z)\in U_i\subseteq B(y,\frac{\varepsilon}{2})$, which is a contradiction to the definition of $m$. Now, as the set $\{n\in\mathbb{N}: f_1^n(z)\in B(y,\frac{\varepsilon}{2})\}$ is infinite, there exists $k\in \mathbb{N}$ such that $f_1^{r_0+k}(z)\in B(y,\frac{\varepsilon}{2})$. Hence by triangle inequality, $d(f^k(f_1^{r_0}(z)),y)<\varepsilon$ or $f^k(x)\in B(y,\varepsilon)$. As the proof holds for any choice of $\varepsilon$ and $y\in X$, the orbit of $x$(under $f$) is dense in $X$. As the proof holds for any $x\in X$, $(X,f)$ is minimal.
\end{proof}

\begin{theorem}\label{tt}
Suppose that $X$ has no isolate point. Let $(X,f_{1,\infty})$ be a non-autonomous system with $f_{1,\infty}$ being feeble open and surjective and converging uniformly to a map $f$. If $\{f_n^{n+k}\}_{k\in\mathbb{N}}$ converges collectively to $\{f^k\}_{k\in\mathbb{N}}$, then $(X,f)$ is transitive $\Leftrightarrow$ $(X,f_{1,\infty})$ is transitive.
\end{theorem}
\begin{proof}
Necessity is proved in Theorem 2.4 by Sharma and Raghav\cite{q36}, we will prove sufficiency.
Let $\epsilon>0$ be given and let $B(x,\epsilon)$ and $B(y,\epsilon)$ be two non-empty open sets in $X$. As $\{f_n^{n+k}\}_{k\in\mathbb{N}}$ converges collectively to $\{f^k\}_{k\in\mathbb{N}}$, there exists $r_0\in\mathbb{N}$ such that $D(f_r^k,f^k) < \frac{\varepsilon}{2}$ for all $r\geq r_0$ and for every $k\in\mathbb{N}$. By Lemma \ref{yinlijmaa2} and transitivity of $(X,f_{1,\infty})$, for any non-empty open sets $U,V$ the set $N_{f_{1,\infty}}(U,V)$ is infinite. Take $U=f_1^{-r_0}(B(x,\epsilon))$ and $V=S(y,\frac{\epsilon}{2})$. Then there exists $k$ such that $f_1^{r_0+k}(U)\cap V\neq\emptyset$. Consequently, there exists $u\in U$ such that $d(f_1^{r_0+k}(u),y)<\frac{\epsilon}{2}$. Hence by triangle inequality, $d(y,f^k(f_1^{r_0}(u)))<\epsilon$. As $f_1^{r_0}(u)\in B(x,\epsilon)$ we have $f^k(B(x,\epsilon))\cap B(y,\epsilon)\neq\emptyset$. As the proof holds for any choice of non-empty open sets $B(x,\epsilon)$ and $B(y,\epsilon)$, the system $(X,f)$ is transitive.
\end{proof}

Theorem \ref{min} and Theorem \ref{tt} show that minimality and transitivity coincide for the $(X,f_{1,\infty})$ and limiting system $(X,f)$, we will prove the result is also true for weakly mixing and total transitivity but is not true for mildly mixing. Besides, compared with Theorem \ref{min} and Theorem \ref{tt}, we can drop the condition '$X$ has no isolate point'.
\begin{theorem}\label{weakly}
Let $(X,f_{1,\infty})$ be a non-autonomous system with $f_{1,\infty}$ being feeble open and surjective and converging uniformly to a map $f$. If $\{f_n^{n+k}\}_{k\in\mathbb{N}}$ converges collectively to $\{f^k\}_{k\in\mathbb{N}}$, then $(X,f_{1,\infty})$ is weakly mixing if and only if $(X,f)$ is weakly mixing.
\end{theorem}
\begin{proof}
Let $\epsilon>0$ be given and let $U_1=B(x_1,\epsilon)$, $U_2=B(x_2,\epsilon)$, $V_1=B(y_1,\epsilon)$ and $V_2=B(y_2,\epsilon)$ be four non-empty open sets in $X$. As $\{f_n^{n+k}\}_{k\in\mathbb{N}}$ converges collectively to $\{f^k\}_{k\in\mathbb{N}}$, there exists $r_0\in \mathbb{N}$ such that $D(f_r^{r+k},f^k) < \frac{\epsilon}{2}$ for all $r\geq r_0$ and for every $k\in\mathbb{N}$. Since $f_{1,\infty}$ is feeble open, $int(f_1^{r_0}(U_1))$ and $int(f_1^{r_0}(U_2))$ are non-empty open of $X$ and thus by weakly mixing of $(X,f)$, for open sets $U_1'=int(f_1^{r_0}(U_1))$,$U_2'=int(f_1^{r_0}(U_2))$,$V_1'=B(y_1,\frac{\epsilon}{2})$ and $V_2'=B(y_2,\frac{\epsilon}{2})$ there exists $m\in \mathbb{N}$ such that $f^m(U_1')\cap V_1' \neq\emptyset$ and $f^m(U_2')\cap V_2' \neq\emptyset$. Consequently there exist $u_1'\in U_1'$ and $u_2'\in U_2'$ such that $f^m(u_1')\in V_1'$ and $f^m(u_2')\in V_2'$. Further, as $U_1'=int(f_1^{r_0}(U_1))$ and $U_2'=int(f_1^{r_0}(U_2))$, there exist $u_1\in U_1$ and $u_2\in U_2$ such that $u_1'=f_1^{r_0}(u_1)$ and $u_2'=f_1^{r_0}(u_2)$, hence $f^m (f_1^{r_0}(u_1))\in V_1'$ and $f^m (f_1^{r_0}(u_2))\in V_2'$.
Also, collective convergence ensures $d(f_1^{m+r_0}(u_1),f^m(f_1^{r_0}(u_1)))<\frac{\epsilon}{2}$ and $d(f_1^{m+r_0}(u_2),f^m(f_1^{r_0}(u_2)))<\frac{\epsilon}{2}$, and hence by the triangle inequality $d(y_1,f_1^{m+r_0}(u_1))<\epsilon$ and $d(y_2,f_1^{m+r_0}(u_2))<\epsilon$, therefore, $f_1^{m+r_0}(U_1)\cap V_1\neq\emptyset$ and $f_1^{m+r_0}(U_2)\cap V_2\neq\emptyset$. As the proof holds for any pair of non-empty open sets $B(x_1,\epsilon),B(x_2,\epsilon),B(y_1,\epsilon),B(y_2,\epsilon)$ in $X$, the proof holds for any pair of non-empty open sets in $X$. Hence $(X,f_{1,\infty})$ is weakly mixing.

Conversely, let $\epsilon>0$ be given and let $B(x_1,\epsilon)$,$B(x_2,\epsilon)$,$B(y_1,\epsilon)$ and $B(y_2,\epsilon)$ be four non-empty open sets in $X$. As $\{f_n^{n+k}\}_{k\in\mathbb{N}}$ converges collectively to $\{f^k\}_{k\in\mathbb{N}}$, there exists $r_0\in \mathbb{N}$ such that $D(f_r^{r+k},f^k) < \frac{\epsilon}{2}$ for all $r\geq r_0$ and for every $k\in\mathbb{N}$. Put $U_1=f_1^{-r_0}(B(x_1,\epsilon))$, $U_2=f_1^{-r_0}(B(x_2,\epsilon))$, $V_1=S(y_1,\frac{\epsilon}{2})$ and $V_2=S(y_2,\frac{\epsilon}{2})$. Applying theorem \ref{weaklyinf} we can choose $k\in\mathbb{N}$ such that $f_1^{r_0+k}(U_1)\cap V_1\neq\emptyset$ and $f_1^{r_0+k}(U_2)\cap V_2\neq\emptyset$. Consequently, there exist $u_1\in U_1$ and $u_2\in U_2$ such that $d(f_1^{r_0+k}(u_1),y_1)<\frac{\epsilon}{2}$ and $d(f_1^{r_0+k}(u_2),y_2)<\frac{\epsilon}{2}$. Further, collective convergence ensures $d(f_1^{k+r_0}(u_1),f^k(f_1^{r_0}(u_1)))<\frac{\epsilon}{2}$ and $d(f_1^{k+r_0}(u_2),f^k(f_1^{r_0}(u_2)))<\frac{\epsilon}{2}$, and hence by triangle inequality we have $d(y_1,f^k(f_1^{r_0}(u_1)))<\epsilon$ and $d(y_2,f^k(f_1^{r_0}(u_2)))<\epsilon$. As $f_1^{r_0}(u_1)\in B(x_1,\epsilon)$ and $f_1^{r_0}(u_2)\in B(x_2,\epsilon)$, we have $f^k(B(x_1,\epsilon))\cap B(y_1,\epsilon)\neq\emptyset$ and $f^k(B(x_2,\epsilon))\cap B(y_2,\epsilon)\neq\emptyset$. As the proof holds for any choice of non-empty open sets $B(x_1,\epsilon),B(x_2,\epsilon),B(y_1,\epsilon),B(y_2,\epsilon)$ in $X$, the system $(X,f)$ is weakly mixing.
\end{proof}

\begin{theorem}\label{totally cd}
Let $(X,f_{1,\infty})$ be a non-autonomous system with $f_{1,\infty}$ being feeble open and surjective and converging uniformly to a map $f$. If $\{f_n^{n+k}\}_{k\in\mathbb{N}}$ converges collectively to $\{f^k\}_{k\in\mathbb{N}}$, then $(X,f_{1,\infty})$ is totally transitive if and only if $(X,f)$ is totally transitive.
\end{theorem}
\begin{proof}
Let $s\in\mathbb{N}$ and $\epsilon>0$ be given and let $U=B(x,\epsilon)$, $V=B(y,\epsilon)$ be two non-empty open sets in $X$. As $\{f_n^{n+k}\}_{k\in\mathbb{N}}$ converges collectively to $\{f^k\}_{k\in\mathbb{N}}$, there exists $r_0\in \mathbb{N}$ such that $D(f_r^{r+k},f^k) < \frac{\epsilon}{2}$ for all $r\geq r_0$ and for every $k\in\mathbb{N}$. Since $f_{1,\infty}$ is feeble open, $int(f_1^{sr_0}(U))$ is non-empty open of $X$ and thus by total transitivity of $(X,f)$, for open sets $U'=int(f_1^{sr_0}(U_1))$ and $V'=B(y,\frac{\epsilon}{2})$ there exists $m\in \mathbb{N}$ such that $f^{sm}(U')\cap V' \neq\emptyset$. Consequently there exist $u'\in U'$ such that $f^{sm}(u')\in V'$. Further, as $U'=int(f_1^{sr_0}(U))$, there exist $u\in U$ such that $u'=f_1^{sr_0}(u_1)$, hence $f^{sm} (f_1^{r_0}(u))\in V'$.
Also, collective convergence ensures $d(f_1^{sm+sr_0}(u),f^{sm}(f_1^{r_0}(u)))<\frac{\epsilon}{2}$, and hence by the triangle inequality $d(y,f_1^{sm+sr_0}(u))<\epsilon$, therefore, $f_1^{sm+sr_0}(U_1)\cap V_1\neq\emptyset$. As the proof holds for any pair of non-empty open sets $B(x,\epsilon),B(y,\epsilon)$ in $X$, the proof holds for any pair of non-empty open sets in $X$. Hence $(X,f_{1,\infty})$ is totally transitive.

Conversely, let $s\in\mathbb{N}$ and $\epsilon>0$ be given and let $B(x,\epsilon)$ and $B(y,\epsilon)$ be two non-empty open sets in $X$. As $\{f_n^{n+k}\}_{k\in\mathbb{N}}$ converges collectively to $\{f^k\}_{k\in\mathbb{N}}$, there exists $r_0\in\mathbb{N}$ such that $D(f_r^k,f^k) < \frac{\varepsilon}{2}$ for all $r\geq r_0$ and for every $k\in\mathbb{N}$. Take $U=f_1^{-sr_0}(B(x,\epsilon))$ and $V=S(y,\frac{\epsilon}{2})$. By total transitivity of $(X,f_{1,\infty})$, for any non-empty open sets $U,V$ we will claim that the set $N_{f_{1,\infty}^{[s]}}(U,V)$ is infinite. If not, let $m=max\{N_{f_{1,\infty}^{[s]}}(U,V)\}$. Then take $s_0=(m+1)s$. Since $f_{1,\infty}$ is totally transitive, for $s_0$ and non-empty open sets $U,V$, there exists $l\in\mathbb{N}$ such that $f_1^{ls_0}(U)\cap V=f_1^{l(m+1)s}(U)\cap V\neq\emptyset$, which implies $l(m+1)\in N_{f_{1,\infty}^{[s]}}(U,V)$. Note that $l(m+1)>m$, so it is a contradiction to the definition of $m$. Now, since $N_{f_{1,\infty}^{[s]}}(U,V)$ is infinite, there exists $k$ such that $f_1^{r_0s+ks}(U)\cap V\neq\emptyset$. Consequently, there exists $u\in U$ such that $d(f_1^{r_0s+ks}(u),y)<\frac{\epsilon}{2}$. Hence by triangle inequality, $d(y,f^{ks}(f_1^{sr_0}(u)))<\epsilon$. As $f_1^{sr_0}(u)\in B(x,\epsilon)$ we have $f^{ks}(B(x,\epsilon))\cap B(y,\epsilon)\neq\emptyset$. As the proof holds for any choice of non-empty open sets $B(x,\epsilon)$ and $B(y,\epsilon)$, the system $(X,f)$ is totally transitive.

\end{proof}

A similar argument establishes equivalence of syndetical transitivity for the two system, so we get the following corollary.

\begin{corollary}\label{tl2}
Let $(X,f_{1,\infty})$ be a non-autonomous system with $f_{1,\infty}$ being feeble open and surjective and converging uniformly to a map $f$. If $\{f_n^{n+k}\}_{k\in\mathbb{N}}$ converges collectively to $\{f^k\}_{k\in\mathbb{N}}$, then $(X,f_{1,\infty})$ is syndetically transitive if and only if $(X,f)$ is syndetically transitive.
\end{corollary}

\begin{theorem}\label{mildly}
Let $(X,f_{1,\infty})$ be a non-autonomous system with $f_{1,\infty}$ being feeble open and  surjective and converging uniformly to a map $f$. If $\{f_n^{n+k}\}_{k\in\mathbb{N}}$ converges collectively to $\{f^k\}_{k\in\mathbb{N}}$, then $(X,f_{1,\infty})$ is mildly mixing $\Rightarrow$ $(X,f)$ is mildly mixing.  Converse is not true in general.
\end{theorem}
\begin{proof}
For any autonomous system $(Y,d_1)$ and for any transitive map $g:Y\rightarrow Y$, let $g_{1,\infty}=\{g,g,...\}$, let $\epsilon>0$ be given, and let $U=B(x,\epsilon)$ and $V=B(y,\epsilon)$ be two non-empty open sets in $X$, and let $U_1$ and $V_1$ be two non-empty open sets in $Y$, then $U\times U_1$ and $V\times V_1$ are two non-empty open sets in $X\times Y$. As $\{f_n^{n+k}\}_{k\in\mathbb{N}}$ converges collectively to $\{f^k\}_{k\in\mathbb{N}}$, there exists $r_0\in \mathbb{N}$ such that $D(f_r^{r+k},f^k) < \frac{\epsilon}{2}$ for all $r\geq r_0$ and for every $k\in\mathbb{N}$. Further, since $f_{1,\infty}$ is mildly mixing, $f_{1,\infty}$ is mixing by Theorem \ref{mild}. Besides, since $N_{g_{1,\infty}}(U_1,V_1)$ is infinite, the set of times when any non-empty open set $U\times U_1$ visits $V\times V_1$ is infinite. Applying transitivity to open sets $U\times U_1'=f_1^{-r_0}(B(x,\epsilon))\times f_1^{-r_0}(U_1)$ and $V\times V_1=B(y,\frac{\epsilon}{2})\times V_1$, we can choose $k$ such that $f_1^{r_0+k}(U)\cap V\neq\emptyset$ and $g_1^{r_0+k}(U_1')\cap V_1\neq\emptyset$. Consequently, there exist $u\in U$ and $u_1\in U_1'$ such that $d(f_1^{r_0+k}(u),y)<\frac{\epsilon}{2}$ and $g_1^{r_0+k}(u_1)\in V_1$. Further, collective convergence ensures $d(f_1^{k+r_0}(u),f^k(f_1^{r_0}(u)))<\frac{\epsilon}{2}$ and hence by triangle inequality $d(y,f^k(f_1^{r_0}(u)))<\epsilon$. As $f_1^{r_0}(u)\in B(x,\epsilon)$ and  $f_1^{r_0}(u_1)\in U_1$ we have $f^k(B(x,\epsilon))\cap B(y,\epsilon)\neq\emptyset$ and $g^k(U_1)\cap V_1\neq\emptyset$. As the proof holds for any choice of non-empty open sets $B(x,\epsilon),B(y,\epsilon)$ of $X$ and $U_1,V_1$ of $Y$, the system $(X,f)$ is mildly mixing.

Conversely, let $f$ be mildly mixing but not mixing in autonomous system. Consider the non-autonomous discrete system $(X,f_{1,\infty})$, where $f_{1,\infty}$ is defined by
$$f_{1,\infty}=\{f,f,f,...\}.$$ Then it is easy to see that $f_{1,\infty}$ is feeble open and surjective, and it converges uniformly to the map $f$. Moreover, $\{f_n^{n+k}\}_{k\in\mathbb{N}}$ converges collectively to $\{f^k\}_{k\in\mathbb{N}}$. We claim that $(X,f_{1,\infty})$ is not mildly mixing, if not, then by Theorem \ref{mild}, $(X,f_{1,\infty})$ is mixing. Hence by \cite[Theorem 2.5]{q36} $f$ is mixing, which is a contradiction.

\end{proof}
\begin{remark}
The above result that mildly mixing of $(X,f_{1,\infty})$ implies mildly mixing of $(X,f)$ can be quickly obtained by Theorem \ref{mild} and \cite[Theorem 2.5]{q36}.
\end{remark}

For autonomous discrete dynamical system it is known that syndetical transitivity and weakly mixing imply multi-transitivity\cite{q33}. However, it is not always true for non-autonomous discrete dynamical system as justified by the example below.

\begin{example}\label{li2}

Let $\Sigma_2$ and $\sigma$ be defined in Example \ref{li1}. Consider the non-autonomous discrete system $(\Sigma_2,f_{1,\infty})$, where $f_{1,\infty}$ is defined by
$$f_{1,\infty}=\{\sigma,\sigma^{-1},\sigma^2,\sigma^{-2},...,\sigma^n,\sigma^{-n},...\}.$$

Now we show that $f_{1,\infty}$ is syndetically transitive. Let $U,V$ be any nonempty open subsets of $X$. Since $\sigma$ is mixing, there exists $M\in\mathbb{N}$ such that $\sigma^p(U)\cap V\neq\emptyset$ for any $p>M$. Notice that $f_1^{2m-1}=\sigma^m$ for any $m\in\mathbb{N}$, so $f_1^{2m-1}(U)\cap V=\sigma^m(U)\cap V\neq\emptyset$ for any $m>M$. Therefore, $f_{1,\infty}$ is syndetically transitive.

Next, we prove that $f_{1,\infty}$ is weakly mixing. Since $\sigma$ is weakly mixing and $f_1^{2m-1}=\sigma^m$ for any $m\in\mathbb{N}$, it is easy to see that $f_{1,\infty}$ is weakly mixing.

Finally, we verify that $f_{1,\infty}$ is not multi-transitive. Take $m=2$, there exist nonempty open subsets $U_1,U_2,V_1,V_2$ of $X$ such that $U_1\cap V_1=\emptyset$ and $U_2\cap V_2=\emptyset$. Then for any $t\in\mathbb{N}$, $f_1^{2t}(U_2)\cap V_2=U_2\cap V_2=\emptyset$. Therefore, $f_{1,\infty}$ is not multi-transitive.

\end{example}
Now, we give a sufficient condition under which syndetical transitivity and weakly mixing imply multi-transitivity in non-autonomous discrete dynamical system.
\begin{theorem}
Let $(X,f_{1,\infty})$ be a non-autonomous system with $f_{1,\infty}$ being feeble open and surjective and converging uniformly to a map $f$. If $\{f_n^{n+k}\}_{k\in\mathbb{N}}$ converges collectively to $\{f^k\}_{k\in\mathbb{N}}$, then syndetical transitivity and weakly mixing of $(X,f_{1,\infty})$ imply multi-transitivity.
\end{theorem}

\begin{proof}
Since $f_{1,\infty}$ is syndetically transitive and weakly mixing, so by Theorem \ref{weakly} and Corollary \ref{tl2}, $f$ is
syndetically transitive and weakly mixing. Thus $f$ is  multi-transitivity by \cite{q33}. Then, by \cite[Theorem4.1]{wx1}, $(X,f_{1,\infty})$ is multi-transitivity.
\end{proof}

For autonomous discrete dynamical system it is known that if $f,g$ are syndetically transitive and weakly mixing, then $f\times g$ is also syndetically transitive and weakly mixing\cite{q33}. However, the result is not always true for non-autonomous discrete dynamical system as justified by following example.

\begin{example}

Let $\Sigma_2$ and $\sigma$ be defined in Example \ref{li1}. Consider the non-autonomous discrete systems $(\Sigma_2,f_{1,\infty})$ and $(\Sigma_2,g_{1,\infty})$, where $f_{1,\infty}$ is defined by
$$f_{1,\infty}=\{\sigma,\sigma^{-1},\sigma^2,\sigma^{-2},...,\sigma^n,\sigma^{-n},...\}$$
and $g_{1,\infty}$ is defined by
$$g_{1,\infty}=\{id,\sigma,\sigma^{-1},\sigma^2,\sigma^{-2},...,\sigma^n,\sigma^{-n},...\}.$$

 With a similar argument in Example \ref{li2}, $f_{1,\infty}$  and $g_{1,\infty}$  syndetically transitive and weakly mixing.

Next, we show that $f_{1,\infty}\times g_{1,\infty}$ is neither syndetically transitive nor weakly mixing. Observe that for any $m\in\mathbb{N}$, $f_1^{2m-1}=\sigma^m$, $f_1^{2m}=id$,  $g_1^{2m-1}=id$ and  $g_1^{2m}=\sigma^m$. Then it is easy to check that $f_{1,\infty}^n\times g_{1,\infty}^n(U,V)=\emptyset$ for any nonempty open subsets $U,V$ of $X$ and any $n\in\mathbb{N}$. The proof is complete.

\end{example}
Now, we provide a sufficient condition under which that $f_{1,\infty},g_{1,\infty}$ are syndetically transitive and weakly mixing imply that $f_{1,\infty}\times g_{1,\infty}$ is syndetically transitive and weakly mixing.
\begin{theorem}
Let $(X,f_{1,\infty}),(Y,g_{1,\infty})$ be non-autonomous systems with $f_{1,\infty},g_{1,\infty}$ being feeble open, surjective, $f_{1,\infty}$ converging uniformly to map $f$ and $g_{1,\infty}$ converging uniformly to map $g$. Suppose that $\{f_n^{n+k}\}_{k\in\mathbb{N}}$ converges collectively to $\{f^k\}_{k\in\mathbb{N}}$ and  that $\{g_n^{n+k}\}_{k\in\mathbb{N}}$ converges collectively to $\{g^k\}_{k\in\mathbb{N}}$. If $f_{1,\infty},g_{1,\infty}$ are both syndetically transitive and weakly mixing, then $f_{1,\infty}\times g_{1,\infty}$ is syndetically transitive and weakly mixing.
\end{theorem}
\begin{proof}
Since $f_{1,\infty}$ and $g_{1,\infty}$ are both syndetically transitive and weakly mixing, so by Theorem \ref{weakly} and Corollary \ref{tl2}, $f$ and $g$ are both
syndetically transitive and weakly mixing. Thus $f\times g$  is  syndetically transitive and weakly mixing by \cite{q33}. Again, using Theorem \ref{weakly} and Corollary \ref{tl2}, $f_{1,\infty}\times g_{1,\infty}$ is syndetically transitive and weakly mixing.
\end{proof}

For autonomous discrete dynamical system it is known that transitivity and (almost)periodic points dense imply syndetical transitivity\cite{q34}. However, the result is not always true for non-autonomous discrete dynamical system as justified by the example below.

\begin{example}

Let $f:S^1\rightarrow S^1$ be the map defined by $f(e^{2\pi i\theta})=e^{2\pi i(\theta+\alpha)}$, where $S^1$ is a unit circle on the complex plane, $\theta\in [0,1]$ and $\alpha$ is a irrational number. Consider the non-autonomous discrete systems $(S,f_{1,\infty})$, where $f_{1,\infty}$ is  given by
\[f_i=\begin{cases}
f^{k}, &  \text{if}\quad i=3^k, \\
f^{-k}, &  \text{if}\quad i=3^k+1,  \\
id, &  else,
\end{cases}\]
where $ k\in \mathbb{N}$, that is

$$f_{1,\infty}=\{f_1,f_2,...,f_n,...\}=f_{1,\infty}=\{id,id,f,f^{-1},id,id,id,id,f^2,f^{-2},id...\}.$$

It is easy to verify that $f_{1,\infty}$ is  transitive. Since $f_1^{2n}=id$ for any $n\in\mathbb{N}$, each $x\in S$ is 2-period point. Hence, $S$  has dense periodic points. But it is easy to see that $f_{1,\infty}$ is not syndetically transitive.

\end{example}
Now, we establish a sufficient condition under which transitivity and periodic points dense imply syndetical transitivity for non-autonomous discrete dynamical system.
\begin{theorem}
Let $(X,f_{1,\infty})$ be a non-autonomous system with $f_{1,\infty}$ being feeble open and surjective and converging uniformly to a map $f$. If $\{f_n^{n+k}\}_{k\in\mathbb{N}}$ converges collectively to $\{f^k\}_{k\in\mathbb{N}}$, then $f_{1,\infty}$ is transitive and has dense periodic points $\Rightarrow$ $f_{1,\infty}$ is syndetically transitive.
\end{theorem}

\begin{proof}
Since $f_{1,\infty}$ is transitive and has dense periodic points, then by Theorem \ref{tt} and \cite[Theorem2.1]{q36}, $f$ is transitive and has dense periodic points. Thus $f$  is  weakly mixing by \cite{q34}. Then, by Corollary \ref{tl2} again, $f_{1,\infty}$ is syndetically transitive.
\end{proof}

For autonomous discrete dynamical system it is known that multi-sensitivity implies thick sensitivity and that thick sensitivity and transitivity imply multi-sensitivity\cite{q31}. However, the results are not always true for non-autonomous discrete dynamical system as justified by the next example.

\begin{example}

Let $\Sigma_2$ and $\sigma$ be defined in Example \ref{li1}. Consider the non-autonomous discrete systems $(\Sigma_2,f_{1,\infty})$, where $f_{1,\infty}$ is defined by
$$f_{1,\infty}=\{\sigma,\sigma^{-1},\sigma^2,\sigma^{-2},...,\sigma^n,\sigma^{-n},...\}.$$
Since $\sigma$ is multi-sensitive and $f_1^{2k-1}=\sigma^k$, $f_{1,\infty}$ is multi-sensitive.  But note that $f_1^{2k}=id$ for any $k$, so $f_{1,\infty}$ is not thick sensitive.

For another thing, let $(X,f)$ is sensitive and transitive but not multi-sensitive. Then Consider the non-autonomous discrete systems $(X,g_{1,\infty})$, where $g_{1,\infty}$ is defined by
$$g_{1,\infty}=\{f,id,f,id,id,f,id,id,id,f,id,id,id,id,...,f,\underbrace{id,id,...id}_{n-fold},...\}.$$

Then it is easy to see that $(X,g_{1,\infty})$ is thickly sensitive and transitive but not multi-sensitive.

\end{example}
Now, we give a sufficient condition under which multi-sensitivity implies thick sensitivity and that thick sensitivity and transitivity imply multi-sensitivity for non-autonomous discrete dynamical system.
\begin{theorem}
Let $(X,f_{1,\infty})$ be a non-autonomous system with $f_{1,\infty}$ being feeble open and  surjective and converging uniformly to a map $f$. If $\{f_n^{n+k}\}_{k\in\mathbb{N}}$ converges collectively to $\{f^k\}_{k\in\mathbb{N}}$, then $f_{1,\infty}$ is multi-sensitive $\Rightarrow$ $f_{1,\infty}$ is thick sensitivity, and $f_{1,\infty}$ is thickly sensitive and transitive $\Rightarrow$ $f_{1,\infty}$ is multi-sensitive.
\end{theorem}
\begin{proof}
Since $f_{1,\infty}$ is multi-sensitive, so by \cite[Theorem 4.9]{q35}, $f$ is multi-sensitive. Thus $f$  is  thickly sensitive by \cite{q31}. Then, it follows that $f_{1,\infty}$ is thickly sensitive by \cite[Theorem 4.6]{q35} .

Conversely, since $f_{1,\infty}$ is thickly sensitive and transitive, so by \cite{q35} and Theorem \ref{tt}, $f$ is thickly sensitive and transitive. Thus $f$  is  multi-sensitive by \cite{q31}. Then, by \cite[Theorem 4.9]{q35}, $f_{1,\infty}$ is multi-sensitive.
\end{proof}

For autonomous discrete dynamical system it is known that total transitivity and periodic points dense imply weakly mixing\cite{q30}.
Now, we provide a sufficient condition under which total transitivity and periodic points dense imply weakly mixing for non-autonomous discrete dynamical system.

\begin{theorem}\label{tpw}
Let $(X,f_{1,\infty})$ be a non-autonomous system with $f_{1,\infty}$ being feeble open and surjective and converging uniformly to a map $f$. If $\{f_n^{n+k}\}_{k\in\mathbb{N}}$ converges collectively to $\{f^k\}_{k\in\mathbb{N}}$, then $f_{1,\infty}$ is totally transitive and has dense periodic points $\Rightarrow$ $f_{1,\infty}$ is weakly mixing.
\end{theorem}
\begin{proof}
Since $f_{1,\infty}$ is totally transitive and has dense periodic points, so by Corollary \ref{tl2} and \cite[Theorem2.1]{q36}, $f$ is totally transitive and has dense periodic points. Thus $f$  is  weakly mixing by \cite{q30}. Then, by Theorem \ref{weakly}, $f_{1,\infty}$ is weakly mixing.
\end{proof}

For autonomous discrete dynamical system it is known that if $f$ is syndetically transitive but not minimal, then it is  syndetically sensitive\cite{q11}. The following theorem show that the result holds for non-autonomous discrete system whenever $f_{1,\infty}$ converges uniformly to a map $f$.
\begin{theorem}\label{sycm}
Let $(X,f_{1,\infty})$ be a non-autonomous system with $f_{1,\infty}$ converging uniformly to a map $f$. If $f_{1,\infty}$ is syndetically transitive but not minimal, then $f_{1,\infty}$ is  syndetically sensitive.
\end{theorem}

\begin{proof}
Let $a\in X$ be such that $orb(a,f_{1,\infty})$ is not dense in $X$. Let $b\in X\setminus \overline{orb(a,f_{1,\infty})}$ and put $\delta=\frac{1}{4}d(b,\overline{orb(a,f_{1,\infty})})$. Take $V=B(b,\delta)$. For any nonempty open set $U$ of $X$, since $f_{1,\infty}$ is syndetically transitive, then $N_{f_{1,\infty}}(U,V)$ is syndetic, with say $M_1$ as a bound for the gaps. Since $f_{1,\infty}$ converges uniformly to a map $f$, by Lemma \ref{yizhi}, there exists a open neighborhoods $W$ of $a$ such that $x\in W$ $\Rightarrow$ $d(f_j^i(a),f_j^i(x))<\delta$ for all $i=0,1,2,...,M_1$ and any $j\in \mathbb{N}$. Note that then $d(f_j^i(W),V)\geq2\delta$ for all $i=0,1,2,...,M_1$ and any $j\in \mathbb{N}$, by the choice of $\delta$. Now, $N_{f_{1,\infty}}(U,W)$ is syndetic, with say $M_2$ as a bound for the gaps. We will show that $N_{f_{1,\infty}}(U,\delta)$ is syndetically sensitive with $M_1+M_2$ as a bound for the gaps. For any $n\in\mathbb{N}$ there exist $k\in\{1,2,...,M_2\}$ and $u\in U$ such that $f_1^{k+n}(u)\in W$. Then by the choice of $W$, one has that for each $i\in\{1,2,...,M_1\}$, $d(f_1^{k+n+i}(u),V)\geq2\delta$. Choose $i\in\{1,2,...,M_1\}$ and $u'\in U$ such that $f_1^{k+n+i}(u')\in V$. Then for this particular $i$, we have $d(f_1^{k+n+i}(u),f_1^{k+n+i}(u'))\geq2\delta>\delta$. Since $n\in\mathbb{N}$ is arbitrary and since $k+i\leq M_1+M_2$, the proof is complete.

\end{proof}

The following theorem removes the condition 'the space has no isolated point' in the \cite[Theorem 3.5(5)]{q29}. Besides, we give a simpler proof than the proof of \cite[Theorem 3.5(5)]{q29}.
\begin{theorem}
Let $(X,f_{1,\infty})$ be a non-autonomous system such that each $f_i(i\in\mathbb{N})$ is surjective. If $f_{1,\infty}$ is strongly transitive, then so is $f_{k,\infty}$ for every $k\geq2$. Converse is true if  $f_{1,\infty}$ is feebly open.
\end{theorem}

\begin{proof}
Let $U$ be any nonempty open subset of $X$ and $k\geq2$. Since each $f_i(i\in\mathbb{N})$ is surjective and continuous, $f_1^{-k+1}(U)$ is nonempty open subset of $X$. By strong transitivity of $f_{1,\infty}$, there exists $M\in \mathbb{N}$ such that
$$\bigcup_{i=1}^Mf_1^i(f_1^{-k+1}(U))=X.$$
That is,
\begin{equation}\label{total}
\bigcup_{i=1}^{M-k}f_{k}^i(U)=X.
\end{equation}
We can assume above $M>k$ because of the surjection of each $f_i$. (\ref{total}) implies that $f_{k,\infty}$ is strongly transitive.

Conversely, let $U$ be any nonempty open subset of $X$ and $k\geq2$. Since $f_{1,\infty}$ is feebly open, $int(f_1^{k-1}(U))$ is nonempty open subset of $X$. By strong transitivity of $f_{k,\infty}$, there exists $M\in \mathbb{N}$ such that
$$\bigcup_{i=1}^Mf_k^i(int(f_1^{k-1}(U)))=X.$$
Therefore, $$X=\bigcup_{i=1}^Mf_k^i(int(f_1^{k-1}(U)))\subseteq \bigcup_{i=1}^Mf_k^i(f_1^{k-1}(U))=\bigcup_{i=k}^{M+k}f_1^i(U)\subseteq\bigcup_{i=1}^{M+k}f_1^i(U).$$

Hence, there exists $M+k$ such that $\bigcup_{i=1}^{M+k}f_1^i(U)=X$. Therefore, $f_{1,\infty}$ is strongly transitive.

\end{proof}

\section{Open Questions}

\noindent\textbf{Question 1.} Suppose that $(X,f_{1,\infty})$ be a non-autonomous system with $f_{1,\infty}$ being feeble open and surjective and converging uniformly to a map $f$. If $\{f_n^{n+k}\}_{k\in\mathbb{N}}$ converges collectively to $\{f^k\}_{k\in\mathbb{N}}$, then whether $(X,f)$ is $\triangle$-transitivity(strong transitivity, respectively) $\Leftrightarrow$ $(X,f_{1,\infty})$ is $\triangle$-transitivity(strong transitivity, almost periodic point, respectively)? Whether $x$ is almost periodic point of $(X,f)$ $\Leftrightarrow$ $x$ is almost periodic point of $(X,f_{1,\infty})$? Or what condition a NDS have to satisfy to ensure the inheritance of $\triangle$-transitivity(strong transitivity, almost periodic point, respectively) to limit map $f$ or vice versa.

\noindent\textbf{Question 2.}  For an autonomous discrete dynamical system, authors have proved that strong transitivity implies  syndetical transitivity\cite{q34}. Does the result hold for a non-autonomous discrete dynamical systems?

\noindent\textbf{Question 3.}  If we drop the condition 'uniform convergence' in Theorem \ref{sycm}('collective convergence' in Theorem \ref{tpw},respectively), does the result still hold in non-autonomous discrete dynamical systems?

\section*{Acknowledgments}
The author would like to thank the the editor and the anonymous referees for their constructive comments and valuable suggestions. This research was supported by the Scientific Research Foundation of Hunan Provincial Education Department(No.23C0148).


\end{document}